\documentclass{article}

\usepackage{fullpage}
\usepackage{amsmath}
\usepackage{amsthm}
\usepackage{amssymb}
\usepackage{graphicx}
\usepackage{mathrsfs}

\newcommand{\CC}{\mathbf C}
\newcommand{\RR}{\mathbf R}
\newcommand{\ZZ}{\mathbf Z}
\newcommand{\NN}{\mathbf N}
\newcommand{\QQ}{\mathbf Q}
\newcommand{\sgn}{\operatorname{sgn}}
\newcommand{\vol}{\operatorname{vol}}

\newcommand{\ch}{\operatorname{ch}}

\newtheorem{thm}{Theorem}[section]
\newtheorem{prop}[thm]{Proposition}
\newtheorem{cor}[thm]{Corollary}
\newtheorem{lemma}[thm]{Lemma}

\theoremstyle{definition}
\newtheorem*{defn}{Definition}

\title{Matching polytopes and Specht modules}
\author{Ricky Ini Liu\\Massachusetts Institute of Technology\\Cambridge, Massachusetts\\\texttt{riliu@math.mit.edu}}

\begin{document}
\maketitle

\begin{abstract}
We prove that the dimension of the Specht module of a forest $G$ is the same as the normalized volume of the matching polytope of $G$. We also associate to $G$ a symmetric function $s_G$ (analogous to the Schur symmetric function $s_\lambda$ for a partition $\lambda$) and investigate its combinatorial and representation-theoretic properties in relation to the Specht module and Schur module of $G$. We then use this to define notions of standard and semistandard tableaux for forests.
\end{abstract}

\section{Introduction}

The \emph{(fractional) matching polytope} of a graph $G$ is the space of all nonnegative edge weightings of $G$ such that the sum of the weights around any vertex is at most 1. This polytope, along with the related \emph{perfect matching polytope}, have been well studied with respect to combinatorial optimization (\cite{KorteVygen}, \cite{LovaszPlummer}).

In this paper, we will first address a natural question, namely to compute the volume of the matching polytope of a forest. We will present several simple recurrences for this volume in Section 2.

We then consider a seemingly unrelated question in the representation theory of the symmetric group. The \emph{Specht module} construction takes the Young diagram of a partition $\lambda$ of $n$ and produces an irreducible representation $S^\lambda$ of the symmetric group $\Sigma_n$. However, this construction can be applied to diagrams of boxes other than Young diagrams of partitions, though it will generally not yield an irreducible representation. For instance, given a skew Young diagram $\lambda/\mu$, one can apply the same construction to get a skew Specht module. Moreover, when one writes the decomposition of this skew Specht module in terms of irreducible representations, the coefficients that appear are the famous Littlewood-Richardson coefficients \cite{JamesPeel}. The general question of describing the Specht module obtained for an arbitrary starting diagram has been the subject of much study (\cite{James}, \cite{ReinerShimozono2}, \cite{ReinerShimozono}, \cite{Magyar}).

Our main result is to show that the two questions described above are in fact related. First we note that the Specht module construction naturally applies to bipartite graphs. We then prove our main result: for a forest $G$, the dimension of the Specht module $S^G$ is exactly the normalized volume $V(G)$ of the matching polytope of $G$.

This result allows us to draw a combinatorial relationship between matchings and Specht modules. In particular, we will ask how to compute the restriction of the Specht module of a forest with $n$ edges from $\Sigma_n$ to $\Sigma_{n-1}$. In the case of partitions, this amounts to a ``branching rule,'' in which one only needs to identify corner boxes of a Young diagram. We will show that for forests there is a similar rule, in which the role of corner boxes is replaced by that of edges in a so-called \emph{almost perfect matching}.

We will also define a symmetric function $s_G$ analogous to the Schur symmetric function $s_\lambda$ for a partition $\lambda$ and describe its combinatorial and representation-theoretic significance. In particular, we will show that it somehow universally describes the recurrence satisfied by $\dim S^G=V(G)$. 

Finally, we will discuss the Schur module construction, which is an analog of the Specht module construction to $GL(N)$-modules. We will then discuss its relation to the previous objects of study. This will allow us to define analogs of standard and semistandard tableaux for forests.

In Section 2, we will discuss the basic properties and recurrences for matching polytopes that we will need for the remainder of the paper. In Section 3, we will introduce Specht modules and prove our main result, as well as deduce a ``branching rule'' for Specht modules of forests. In Section 4, we will define a Schur symmetric function for forests and describe its universality. In Section 5, we will discuss Schur modules and define a notion of semistandard tableaux for forests. Finally, in Section 6 we will formulate some remaining questions and concluding remarks.

\section{Matching polytopes}

In this section, we will describe a recurrence for the normalized volume of the matching polytope of a forest. We will then use this recurrence to provide a characterization for this volume. We will also provide a second, more efficient recurrence for calculating the volume.

We begin with some key definitions.

\begin{defn} Let $G=(V,E)$ be a graph. The \emph{matching polytope} $M_G$ of $G$ is the space of all nonnegative edge weightings $w\colon E \to \RR_{\geq 0}$ such that for all $v \in V$,
\[\sum_{e \ni v} w(e) \leq 1.\]
\end{defn}

Note that $M_G$ is a rational convex polytope of full dimension in $\RR^n$, where $n=|E|$ is the number of edges in $G$. 
The reason that $M_G$ is called a matching polytope is due to the following definition.

\begin{defn} A \emph{matching} $M$ of a graph $G$ is a collection of edges of $G$ such that no two edges of $M$ share a vertex. \end{defn}

Given a matching $M$ of $G$, let us write $\chi_M$ for the edge weighting of $G$ defined by $\chi_M(e)=1$ if $e \in M$ and $0$ otherwise. Then clearly $\chi_M \in M_G$ for all matchings $M$.

The following well known proposition characterizes when $M_G$ is a lattice polytope. (See, for instance, \cite{Edmonds}.)

\begin{prop} The matching polytope $M_G$ is a lattice polytope if and only if $G$ is bipartite. In this case, $M_G$ is the convex hull of $\chi_M$, where $M$ ranges over all matchings of $G$. \end{prop}

It follows that when $G$ is bipartite, the volume of $M_G$ is an integer multiple of $\frac{1}{n!}$. Therefore, we let
\[V(G)=n! \cdot \vol(M_G) \in \ZZ.\]

When considering bipartite graphs, it will be important later that we distinguish the two parts of our graph. Therefore, we will henceforth assume that all bipartite graphs are equipped with a bipartition of the vertices, that is, each vertex will be colored either black or white such that no two adjacent vertices have the same color. We will also assume that our graphs contain no isolated vertices. (Removing any isolated vertices will not change the matching polytope of the graph.)

In particular, let $\mathcal F$ be the set of all finite bipartitioned forests. We will now present the recurrences that will allow us to compute $V(G)$ for any $G \in \mathcal F$.

First, we require the following base case.

\begin{prop}
Let $T_n$ be the star with $n$ edges and white center vertex. Then $V(T_n)=1$.
\end{prop}
\begin{proof}
If the edges have weights $w_1, \dots, w_n$, then $M_G$ is defined by $0 \leq w_i \leq 1$ for all $i$ and $\sum_{i=1}^n w_i \leq 1$. This is just the elementary $n$-simplex in $\RR^n$, which has volume $\frac{1}{n!}$.
\end{proof}

Next, we use the following recurrence to reduce to connected graphs.

\begin{prop}
Let $G$ be a disjoint union of graphs $G_1+G_2$, where $G_1$ has $m$ edges and $G_2$ has $n-m$ edges. Then $V(G)=\binom{n}{m}V(G_1)V(G_2)$.
\end{prop}
\begin{proof}
Clearly $M_G=M_{G_1} \times M_{G_2}$. Hence
\begin{align*}
V(G)&=n!\cdot \vol(M_G)\\
&=n! \cdot \vol(M_{G_1}) \cdot \vol(M_{G_2})\\
&=\frac{n!}{m! \cdot (m-n)!} \cdot V(G_1)\cdot V(G_2).\qedhere
\end{align*}
\end{proof}

We now present a more interesting recurrence for $V(G)$.

Let $H$ be a graph, and let $v_1$ and $v_2$ be distinct vertices of $H$. We construct three graphs $G$, $G_1$, and $G_2$ as follows. Let $G$ be the graph obtained from $H$ by adding pendant edges $\overline{v_1v_1'}$ and $\overline{v_2v_2'}$. (Then $v_1'$ and $v_2'$ are leaves in $G$.) Let $G_1$ be the graph obtained from $H$ by adding a pendant edge $\overline{v_1v_1'}$ and an edge $\overline{v_1v_2}$, and let $G_2$ be obtained from $H$ by adding a pendant edge $\overline{v_2v_2'}$ and an edge $\overline{v_1v_2}$. (See Figure~\ref{figure-leaf}.)

 \begin{figure}
 \centering
 \includegraphics[width=12cm]{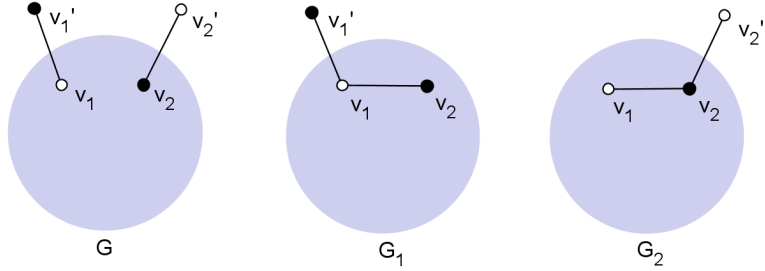}
 \caption{Three graphs as related by the leaf recurrence. They differ only in the edges marked; inside the circles the graphs are arbitrary as long as they are the same in all three cases.}
 \label{figure-leaf}
 \end{figure}

\begin{prop}[Leaf recurrence]\label{leaf}
Let $G$, $G_1$, and $G_2$ be as described above. Then $V(G)=V(G_1)+V(G_2)$.
\end{prop}
\begin{proof}
Consider the matching polytope $M_G$, and write $w_i$ for the weight of $\overline{v_iv_i'}$. Let $M_G^1$ be the intersection of $M_G$ with the halfspace $w_1 \geq w_2$, and let $M_G^2$ be the intersection of $M_G$ with the halfspace $w_2 \geq w_1$. Clearly $\vol(M_G)=\vol(M_G^1)+\vol(M_G^2)$.

Consider the matching polytope $M_{G_1}$, and write $z_1$ for the weight of $\overline{v_1v_1'}$ and $z_2$ for the weight of $\overline{v_1v_2}$. Then for any $z \in M_{G_1}$, note that by letting $w_1=z_1+z_2$, $w_2=z_2$, and keeping all other weights the same, we obtain a point $f(z) \in M_G^1$. Moreover, $f$ is a bijection from $M_{G_1}$ to $M_G^1$: the inverse map is given by letting $z_1=w_1-w_2$, $z_2=w_2$, and keeping all other weights the same. Since $f$ is a volume-preserving linear transformation, it follows that $\vol(M_{G_1})=\vol(M_G^1)$.

An analogous argument gives that $\vol(M_{G_2})=\vol(M_G^2)$. The result follows easily.
\end{proof}
Given any function $f$ on forests, we will say that $f$ \emph{satisfies the leaf recurrence} if $f(G)=f(G_1)+f(G_2)$ for any three graphs $G$, $G_1$, and $G_2$ as described above.

We claim that the previous three propositions suffice to calculate $V(G)$ for any forest $G$.

\begin{prop} \label{V-unique}
There is a unique function $f\colon \mathcal F \to \RR$ satisfying the following properties:
\begin{enumerate}
\item For the star $T_n$ with $n$ edges and white center vertex, $f(T_n)=1$.
\item If $G_1$ and $G_2$ have $m$ and $n-m$ edges, respectively, then $f(G_1+G_2)=\binom{n}{m}f(G_1)f(G_2)$.
\item The function $f$ satisfies the leaf recurrence.
\end{enumerate}
In this case, $f(G)=V(G)$, the normalized volume of the matching polytope of $G$.
\end{prop}

\begin{proof}
Since we have seen that $V(G)$ satisfies these three properties, it suffices to show that for any forest $G$, one can determine $f(G)$ using these properties alone.

We induct first on $n$, the number of edges of $G$. By (2), we may then assume that $G$ is connected and hence a tree. Choose a white vertex $v_0$ to be the root of $G$. We will then induct on $s$, the sum of the distances from all vertices to $v_0$. The base case is when $s=n$, which occurs only when $G$ is $T_n$, in which case $f(G)=1$ by (1).

Suppose $s>n$. Then $G$ must have a leaf $v_1'$ whose neighbor is $v_1 \neq v_0$. Let $v_2$ be the neighbor of $v_1$ closest to $v_0$, and let $G'$ be the graph obtained from $G$ by removing the edge $\overline{v_1v_1'}$ and adding a pendant edge $\overline{v_2v_2'}$. Then if $H$ is the forest obtained from $G$ by removing the edge $\overline{v_1v_2}$ and adding a pendant edge $\overline{v_2v_2'}$, we have that $f(H)=f(G)+f(G')$ by (3). Since $H$ is disconnected, we can calculate $f(H)$ by induction. But $s(G')=s(G)-1$, so we can determine $f(G')$ by induction as well. Thus we can determine $f(G)=f(H)-f(G')$, completing the proof.
\end{proof}

Note from the proof of Proposition~\ref{V-unique} that even if we only stipulate that the three conditions hold only when all graphs involved have at most $n$ edges, we still find that $f(G)=V(G)$ for all graphs with at most $n$ edges.

Although the leaf recurrence suffices to calculate $V(G)$ and serves an important purpose in Section 3, it is not very efficient for calculations. Therefore, we will introduce an alternative recurrence for calculating $V(G)$ for $G \in \mathcal F$. In fact, this second recurrence will have important ramifications in Section 3 as well. We first make the following definition.

\begin{defn} We say that a matching $M$ of $G$ is \emph{almost perfect} if every isolated edge of $G$ lies in $M$ and every non-leaf vertex of $G$ is contained in an edge of $M$. \end{defn}

The importance of almost perfect matchings lies in the following proposition.

\begin{prop} \label{almost-perfect} Let $G$ be a bipartite graph, and let $M$ be an almost perfect matching of $G$. Then
\[V(G)=\sum_{e \in M} V(G\backslash e).\] \end{prop}
\begin{proof}
We claim that it is possible to partition $M_G$ into cones $C_e$ of volume $\frac 1{n!}V(G\backslash e)$ for $e \in M$. Note that $M_{G\backslash e}$ embeds naturally in $M_G$: in fact, $M_{G\backslash e}=M_G \cap \{w \mid w(e)=0\}$. For $e \in M$, let $C_e$ for the cone with vertex $\chi_M$ and base $M_{G \backslash e}$. Since $C_e$ has height 1, it has volume $\frac 1n\vol (M_{G \backslash e}) = \frac 1{n!}V(G\backslash e)$, so it suffices to show that the $C_e$ partition $M_G$ (up to a measure zero set).

Let $w \in M_G$, and let $e_0$ be such that $w(e_0)=t$ is minimum among all $w(e)$ for $e \in M$. Let \[w'=\frac{1}{1-t}(w-t \cdot \chi_M),\] so that $w=t \cdot \chi_M + (1-t) \cdot w'$. Clearly $0 \leq t \leq 1$. By our choice of $e_0$, $w'$ is a nonnegative weighting of $G \backslash e_0$. We claim that it lies in $M_G$ (and hence in $M_{G \backslash e_0}$), which will imply that $w \in C_{e_0}$.

Note that although $M_G$ is defined by an inequality for each vertex of $G$, the condition at a leaf that is not part of an isolated edge is redundant, for it is superseded by the condition at the adjacent vertex. Therefore, to check that $w' \in M_G$, it suffices to check that the sum of the weights of the edges incident to any non-leaf is at most $1$ and that the weight on any isolated edge is at most $1$. If $v$ is a non-leaf of $G$, then
\[
\sum_{e \ni v} w'(e)=\frac{1}{1-t}\left(\sum_{e \ni v} w(e)-t \cdot \sum_{e \ni v} \chi_M(e)\right)
=\frac{1}{1-t}\left(\sum_{e \ni v} w(e)-t\right)
\leq 1.
\]
Similarly, if $e$ is an isolated edge of $G$, then $w'(e)=\frac{1}{1-t}(w(e)-t) \leq 1$. It follows that $w' \in M_{G \backslash e_0}$, so $w \in C_{e_0}$. Therefore the cones $C_e$ for $e \in M$ cover $M_G$.

We now show that the $C_e$ have disjoint interiors. Suppose that $w$ lies in the interior of $C_{e_0}$. Then we can write $w = t\cdot \chi_M + (1-t)\cdot w'$, where $0<t<1$ and $w'$ lies in the interior of $M_{G \backslash e_0}$, so $w(e)=t+(1-t)\cdot w'(e)$ for $e \in M$. Since $w'$ lies in the interior of $M_{G \backslash e_0}$, we have that $w'(e_0)=0$, but $w'(e)>0$ for all other $e \in M$. Therefore, $e_0$ is uniquely determined: $w(e_0)$ is the unique minimum among all $w(e)$ for $e \in M$. This proves the result.
\end{proof}

Not all graphs have almost perfect matchings, but every forest does.
\begin{prop}
Every forest has an almost perfect matching.
\end{prop}
\begin{proof}
It suffices to show that every rooted tree $G$ with at least one edge has an almost perfect matching $M$ such that the root lies in an edge of $M$. We induct on the number of edges of $G$, with the base case being trivial. Choose any edge $e$ incident to the root, and consider the forest $G'$ obtained from $G$ by removing both endpoints of $e$ as well as any edge incident to either endpoint. Root each component of $G'$ at the vertex that was closest to $e$ in $G$. By induction, any component with at least one edge has an almost perfect matching containing an edge that contains the root. It is then easy to check that the union of these matchings together with $e$ is an almost perfect matching of $G$.
\end{proof}

It follows that we can use Proposition~\ref{almost-perfect} to recursively compute $V(G)$ for any $G \in \mathcal F$. One way to express this is to say that $V(G)$ counts the number of standard labelings of $G$ in the following sense.

\begin{defn}
Fix an almost perfect matching $M(G)$ for every $G \in \mathcal F$. For any $G \in \mathcal F$, we say that an edge labeling $z\colon E \to \NN$ is \emph{standard} if $z$ is a bijection between $E$ and $[n]=\{1, \dots, n\}$ such that $z^{-1}(n) \in M(G)$ and, if $n>1$, $z|_{G\backslash z^{-1}(n)}$ is a standard labeling of $G \backslash z^{-1}(n)$.
\end{defn}

The following proposition is then immediate.

\begin{prop} \label{standard}
For any $G \in \mathcal F$, $V(G)$ is the number of standard labelings of $G$.
\end{prop}
\begin{proof}
Both $V(G)$ and the number of standard labelings of $G$ satisfy the recurrence in Proposition~\ref{almost-perfect}.
\end{proof}

These standard labelings will turn out to be analogues of standard Young tableaux. We will also define analogues of semistandard Young tableaux in section 5.

Armed with the results of this section, we will now be able to draw the connection between matching polytopes and Specht modules.


\section{Specht modules}

In this section, we first define the notion of a Specht module for diagrams and bipartite graphs. We will also summarize some known results about them. We will then state and prove our main theorem. As an immediate corollary, we will also derive a branching rule for Specht modules of forests.

(For more information on the basic background needed for this section, see, for instance, \cite{Sagan}.)

Consider an array of unit lattice boxes in the plane. We will write $(i,j)$ to denote the box in the $i$th row from the top and the $j$th column from the left, where $i$ and $j$ are positive integers. By a \emph{diagram}, we will mean any subset of these boxes.

For instance, a \emph{partition} $\lambda=(\lambda_1, \lambda_2, \dots)$ of $|\lambda|=n$ is a sequence of weakly decreasing nonnegative integers summing to $n$. (We may add or ignore trailing zeroes to $\lambda$ as convenient.) Then the \emph{Young diagram} of $\lambda$ consists of all boxes $(i,j)$ with $j \leq \lambda_i$. We may refer to a partition and its Young diagram interchangeably. If $\lambda$ and $\mu$ are partitions such that $\mu_i \leq \lambda_i$ for all $i$, then the \emph{skew Young diagram} $\lambda/\mu$ consists of all boxes in $\lambda$ not in $\mu$.

Given a diagram $D$, a \emph{tableau} of shape $D$ is a filling of the boxes of $D$ with positive integers. In the case of a (skew) Young diagram, a tableau $T$ is called a \emph{semistandard Young tableau} if the rows are weakly increasing and the columns are strictly increasing. A semistandard Young tableau is \emph{standard} if it contains only the numbers $1, 2, \dots, n$, each exactly once.

Consider any diagram $D$ of $n$ boxes. Order the boxes of $D$ arbitrarily, and let the symmetric group $\Sigma_n$ act on them in the obvious way. Let $R_D$ be the subgroup containing those $\sigma \in \Sigma_n$ that stabilize each row of $D$, and likewise define $C_D$ for columns of $D$. Let $\CC[\Sigma_n]$ denote the group algebra over $\Sigma_n$, and consider elements  
\[R(D) = \sum_{\sigma \in R_D} \sigma \mbox{ and }
C(D)=\sum_{\sigma \in C_D} \sgn(\sigma)\sigma.\]

\begin{defn}
The \emph{Specht module} (over $\CC$) of $D$ is the left ideal
\[S^D=\CC[\Sigma_n]C(D)R(D).\]
\end{defn}

Clearly $S^D$ is a representation of $\Sigma_n$ by left multiplication. We may interpret basis elements of $\CC[\Sigma_n]$ as bijections from $D$ to $[n]$, or, equivalently, as tableaux of shape $D$ with labels $1, 2, \dots, n$. Then multiplying on the right corresponds to applying a permutation to the boxes of $D$, while the action of $\Sigma_n$ on the left corresponds to applying a permutation to the labels $[n]$.

It is well known that the irreducible representations of $\Sigma_n$ are exactly the Specht modules $S^{\lambda}$, where $\lambda$ ranges over all partitions of $n$. Moreover, the dimension of $S^\lambda$ is $f^\lambda$, the number of standard Young tableaux of shape $\lambda$. 

Given two diagrams $D$ and $E$, let $D \oplus E$ be the diagram obtained by drawing them so that no box of $D$ lies in the same row or column as a box of $E$. In other words, choose some positive integer $N$ such that no box of $D$ has a coordinate larger than $N$, and then let $D \oplus E=D \cup \tilde E$, where $\tilde E$ is the same diagram as $E$ but shifted by the vector $(N,N)$. If $D$ has $m$ boxes and $E$ has $n-m$ boxes, then we can think of $\Sigma_m$ and $\Sigma_{n-m}$ acting on $D$ and $E$ independently (where we think of $\Sigma_m \subset \Sigma_{n}$ as acting on the first $m$ letters and $\Sigma_{n-m} \subset \Sigma_{n}$ as acting on the last $n-m$ letters). Then it follows that
\[S^{D \oplus E}=\operatorname{Ind}_{\Sigma_m \times \Sigma_{n-m}}^{\Sigma_{n}}(S^D \otimes S^E).\]

Let us write $c_\lambda^D$ for the multiplicity of $S^{\lambda}$ in $S^D$, so that
\[S^D \cong \bigoplus (S^\lambda)^{\oplus c_\lambda^D}.\]
In the case when $D$ is a skew shape, the coefficients $c_\lambda^D$ are the famous \emph{Littlewood-Richardson coefficients}, and there are a number of combinatorial ways to calculate them, collectively called \emph{Littlewood-Richardson rules}. However, the question of how to compute $\dim S^D$ or $c_\lambda^D$ for an arbitrary diagram $D$ is still open. The most general known result, due to Reiner and Shimozono \cite{ReinerShimozono},
applies to so-called \emph{percentage-avoiding} diagrams.
\medskip

Although Specht modules are traditionally defined for diagrams, we can also define the Specht module of a bipartite graph. Indeed, observe that permuting the rows or permuting the columns of $D$ does not change the structure of $S^D$ up to isomorphism. Therefore, for any (bipartitioned) bipartite graph $G$, we can define the Specht module $S^G$ as follows.

Number the white vertices of $G$ with positive integers, and likewise number the black vertices. Then let $D=D(G)$ be the diagram that contains the box $(i,j)$ if and only if $G$ contains an edge connecting white vertex $i$ and black vertex $j$. By our observation, the structure of $S^D$ does not depend on the labeling of the vertices of $G$, so we will write $S^G$ for $S^{D(G)}$. We will also sometimes refer to the graph $G$ and the diagram $D$ interchangeably. We will say two diagrams are equivalent if their underlying graphs are the same. (We may also implicitly identify them at times as convenient.)

In other words, we let $\Sigma_n$ be the symmetric group on the edges of $G$. Then $C_G$ corresponds to the subgroup that stabilizes the incident edges to any black vertex and $R_G$ to the subgroup that stabilize the incident edges to any white vertex. We can then form the Specht module $S^G=\CC[\Sigma_n]C(G)R(G)$ exactly as before.

We can now state our main theorem on the Specht module of a forest.

\begin{thm} \label{main-theorem}
For all $G \in \mathcal F$, $\dim S^G = V(G)$.
\end{thm}

It is worth noting that the general forest $G$ is not covered by Reiner and Shimozono's result on percentage-avoiding diagrams \cite{ReinerShimozono}.

To prove this theorem, we will show that $\dim S^G$ satisfies the properties of Proposition~\ref{V-unique}. While the first two properties are easy, the third will require a bit of work.

\medskip

The first ingredient is the following lemma by James and Peel \cite{JamesPeel}.

Given a diagram $D$, choose two boxes $(i_1, j_1), (i_2, j_2) \in D$ such that $(i_1, j_2), (i_2, j_1) \notin D$. We form the diagram $D^A$ as follows. For every row except $i_1$ and $i_2$, $D^A$ contains exactly the same boxes as $D$. Also, $D^A$ contains the box $(i_1,j)$ if and only if $D$ contains both $(i_1,j)$ and $(i_2,j)$, while $D^A$ contains the box $(i_2,j)$ if and only if $D$ contains either $(i_1,j)$ or $(i_2,j)$.

Informally, $D^A$ is obtained from $D$ by shifting boxes from row $i_1$ to row $i_2$ if possible. Let us similarly define $D^B$ using columns $j_1$ and $j_2$. (We will refer to this process as \emph{splitting} $D$ into $D^A$ and $D^B$ using boxes $(i_1,j_1)$ and $(i_2,j_2)$.) We then have the following result, which we present without proof.

\begin{lemma}\label{JP}
There exists a $\CC[\Sigma_n]$-module homomorphism of Specht modules $\varphi \colon S^D \to S^{D^A}$. Moreover, $S^{D^B} \subset S^D$ and $\varphi(S^{D^B})=0$.
\end{lemma}

This allows us to easily prove part of the leaf recurrence for $\dim S^G$.

\begin{prop} \label{easy}
Let $G$, $G_1$, and $G_2$ be three forests as appearing in Proposition~\ref{leaf} (the leaf recurrence). Then $\dim S^G \geq \dim S^{G_1}+\dim S^{G_2}$.
\end{prop}
\begin{proof}
It is easy to check that the diagram of $G$ splits into $G_1$ and $G_2$. Then by Lemma~\ref{JP}, $S^G$ contains a subrepresentation isomorphic to $S^{G_1} \oplus S^{G_2}$. The result follows.
\end{proof}

The second ingredient will be a lower bound on $\dim S^G$ obtained by considering $S^G$ as a representation of $\Sigma_{n-1}$ instead of $\Sigma_n$. We first begin with a definition.

\begin{defn}
We say that a matching $M$ of a bipartite graph $G$ is \emph{special} if there does not exist a cycle of $G$ half of whose edges lie in $M$.
\end{defn}

Note that every matching in a forest is special since there are no cycles.

\begin{defn}
A subset $U$ of boxes in a diagram $D$ is called a \emph{transversal} if no two boxes of $U$ lie in the same row or column.
\end{defn}

Clearly a transversal of a diagram corresponds to a matching of the corresponding bipartite graph.

\begin{prop} \label{special-transversal}
Let $U$ be a transversal of a diagram $D$, and let $E$ be the subdiagram of $D$ given by the intersection of the rows and columns of $D$ containing boxes in $U$. Then the following are equivalent:
\begin{enumerate}
\item The edges corresponding to $U$ in the graph of $D$ form a special matching.
\item The set $U$ is the unique transversal in $E$ of size $|U|$.
\item There exists a diagram equivalent to $E$ such that the boxes corresponding to $U$ lie along the main diagonal and all other boxes of $E$ lie below the main diagonal.
\end{enumerate}
\end{prop}
\begin{proof}
For the first equivalence, let $M$ be the matching corresponding to $U$ in the graph $G$ corresponding to $D$. If $M$ is not special, then there exists a cycle half of whose edges lie in $M$. Replacing those edges in $M$ with the other edges in the cycle gives another matching $M'$, and the corresponding transversal $U'$ lies in the same rows and columns of $D$ as $U$. Conversely, suppose two subsets $U$ and $U'$ are both transversals lying in the same rows and columns. If $M$ and $M'$ are the corresponding matchings, then $(M \backslash M') \cup (M' \backslash M)$ consists of two disjoint matchings on the same set of vertices. It therefore has the same number of edges as vertices, so it contains a cycle. Since both $M$ and $M'$ are matchings, exactly half of the edges in the cycle must lie in each of $M$ and $M'$, so $M$ is not special.

For the second equivalence, suppose $U$ is the unique transversal in $E$ of size $|U|$. We may assume that $U$ consists of the boxes $(i,i)$ for $1 \leq i \leq u$. Consider the directed graph $H$ on $[u]$ with an edge from $i$ to $j$ if $(i,j) \in E\backslash U$. In fact $H$ is acyclic: if $(i_0, i_1, \dots, i_s=i_0)$ formed a cycle in $H$, then we could replace $(i_j, i_j)$ in $U$ by $(i_j, i_{j+1})$ for $0 \leq j < s$ and obtain another transversal in $E$ of size $|U|$. It follows that we can reorder the vertices of $H$ such that there is an edge from $i$ to $j$ only if $i>j$. Applying this reordering to the rows and columns of $E$ gives an equivalent diagram with the desired property. Since the other direction of the equivalence is trivial, this completes the proof.
\end{proof}

We will say that a transversal satisfying the conditions of Proposition~\ref{special-transversal} is \emph{special}. The importance of special transversals comes from the following proposition.

\begin{prop} \label{restrict}
Let $U$ be a special transversal of a diagram $D$. Then the restricted representation $S^D|_{\Sigma_{n-1}}$ contains a subrepresentation isomorphic to $\bigoplus_{x \in U} S^{D \backslash \{x\}}$. In particular,
\[\dim S^D \geq \sum_{x \in U} \dim S^{D \backslash \{x\}}.\]
\end{prop}
\begin{proof}
By Proposition~\ref{special-transversal}, we may assume that $U$ consists of boxes $(i,i)$ for $1 \leq i \leq u$ and that $(i,j) \notin D$ for $1 \leq i < j \leq u$. Write $D_i=D \backslash \{(i,i)\}$.

Recall that we interpret the basis elements of $\CC[\Sigma_n]$ as tableaux of shape $D$ with entries $1, \dots, n$. For $1 \leq i \leq u$, let $T_i$ be the set of tableaux of shape $D$ such that $n$ appears in box $(i,i)$. By our choice of $U$, every term of $TC(D)R(D)$ for $T \in T_i$ contains the entry $n$ in row at least $i$.

Let $\varphi_i$ be the linear map on $\CC[\Sigma_n]$ defined on a tableau $T$ of shape $D$ by $\varphi_i(T)=T$ if $T \in T_i$ and 0 otherwise. Since the action of $\Sigma_{n-1}$ on tableaux in $\CC[\Sigma_n]$ does not change the position of $n$, it follows that $\varphi_i$ is a $\Sigma_{n-1}$-homomorphism.

Let $V_i$ be the subrepresentation of $S^D|_{\Sigma_{n-1}}$ generated by $(T_i \cup T_{i+1} \cup \dots \cup T_u)C(D)R(D)$. By above, we have that for $1 \leq i \leq u$, $V_{i+1} \subset \ker \varphi_i$. We then have that $\varphi_i(V_i)$ is generated by $\varphi_i(T_i C(D)R(D))=T_iC(D_i)R(D_i)$, so it is naturally isomorphic to $S^{D_i}$. Therefore, the successive quotients of the filtration $0 \subset V_u \subset V_{u-1} \subset \dots \subset V_1$ contain $S^{D_u}$, $S^{D_{u-1}}$, \dots, $S^{D_1}$ in succession. The result follows easily.
\end{proof}

Comparing Proposition~\ref{restrict} with Proposition~\ref{almost-perfect}, we obtain the following.
\begin{prop} \label{hard}
For all $G \in \mathcal F$, $\dim S^G \geq V(G)$.
\end{prop}
\begin{proof}
We induct on the number of edges in $G$. Let $M$ be an almost perfect matching of $G$. Then since $G$ is a forest, $M$ is special, so by Proposition~\ref{restrict}, the induction hypothesis, and Proposition~\ref{almost-perfect},
\[\dim S^G \geq \sum_{e \in M} \dim S^{G \backslash e} \geq \sum_{e \in M} V(G \backslash e) = V(G).\qedhere\]
\end{proof}

Armed with these facts, it is a simple matter to deduce Theorem~\ref{main-theorem}.

\begin{proof}[Proof of Theorem~\ref{main-theorem}]
We wish to show that for any $G \in \mathcal F$, $\dim S^G=V(G)$. It suffices to show that $\dim S^G$ satisfies the three conditions given in Proposition~\ref{V-unique}. We induct on $n$, the number of edges of $G$, noting as before that if we prove the conditions for $n \leq k$, then $\dim S^G=V(G)$ for all graphs with at most $k$ edges. Let us then assume that we have proven the claim for all graphs with fewer than $n$ edges.

The first condition states that if $T_n$ is the tree with $n$ edges and white center vertex, then $\dim S^{T_n}=1$. But the diagram of $T_n$ is a row of $n$ boxes, so its Specht module is the trivial representation, which has dimension 1. (This also proves the base case when $n=1$.)

The second condition states that if $G_1$ and $G_2$ have $m$ and $n-m$ edges, respectively, then the disjoint union $G_1+G_2$ satisfies $\dim S^{G_1+G_2} = \binom{n}{m} \cdot \dim S^{G_1} \cdot \dim S^{G_2}$. But as diagrams, $G_1+G_2 = G_1 \oplus G_2$, so
\begin{align*}
\dim S^{G_1+G_2}&=\dim \operatorname{Ind}_{\Sigma_m \times \Sigma_{n-m}}^{\Sigma_{n}}(S^{G_1} \otimes S^{G_2})\\
&= [\Sigma_{n} : \Sigma_m \times \Sigma_{n-m}] \cdot \dim (S^{G_1} \otimes S^{G_2})\\
&= \binom{n}{m} \cdot \dim S^{G_1} \cdot \dim S^{G_2}.
\end{align*}
In particular, this implies by induction that $\dim S^G=V(G)$ for any disconnected graph $G$ with $n$ edges.

The third and final condition states that $\dim S^G$ satisfies the leaf recurrence. In other words, we need to show that for $G, G_1, G_2 \in \mathcal F$ related as in Proposition~\ref{leaf}, $\dim S^G=\dim S^{G_1}+\dim S^{G_2}$. Since $G$ is disconnected, we have from above that $\dim S^G=V(G)$. But by Propositions~\ref{easy}, \ref{hard}, and \ref{leaf}, we have
\[V(G) = \dim S^G \geq \dim S^{G_1}+\dim S^{G_2} \geq V(G_1)+V(G_2)=V(G).\]
Therefore we must have equality everywhere. The result follows immediately.
\end{proof}

Note that equality must hold everywhere in the proof above, and therefore we must also have equality in Propositions~\ref{easy} and \ref{hard}. This gives us the following two corollaries.

\begin{cor} \label{exact-split}
Let $G, G_1, G_2 \in \mathcal F$ be related as in Proposition~\ref{leaf} (the leaf recurrence). Then
\[S^G \cong S^{G_1} \oplus S^{G_2}.\]
\end{cor}

\begin{cor} \label{corners}
Let $G \in \mathcal F$, and let $M$ be an almost perfect matching of $G$. Then
\[S^G|_{\Sigma_{n-1}} \cong \bigoplus_{e \in M} S^{G\backslash e}.\]
\end{cor}

Recall that for a partition $\lambda$, the restriction of the Specht module $S^\lambda$ from $\Sigma_n$ to $\Sigma_{n-1}$ can be described as a direct sum of $S^\mu$, where $\mu$ ranges over all partitions that can be obtained from $\lambda$ by removing a corner box from the Young diagram of $\lambda$. Therefore Corollary~\ref{corners} shows that the edges in an almost perfect matching of a forest are analogous to corner boxes of a Young diagram.


\section{Symmetric functions}

In this section, we will associate to each forest $G$ a symmetric function $s_G$. We will also show that the map sending $G$ to $s_G$ is, in a sense, the most general or universal function satisfying the leaf recurrence.

Recall that to a Young diagram $\lambda$, we may associate a symmetric function
\[s_{\lambda}(x_1, x_2, \dots)=\sum_T x^T,\]
where $T$ ranges over all semistandard tableaux of shape $\lambda$, and $x^T=x_1^{\alpha_1}x_2^{\alpha_2}\cdots$ if $T$ contains $\alpha_i$ occurrences of $i$. The $s_\lambda$ form an additive integer basis for the ring $\Lambda_\ZZ$ of symmetric functions over the $x_i$ with coefficients in $\ZZ$. Moreover, the multiplicative structure constants for the $s_\lambda$ are given by the Littlewood-Richardson coefficients.

Due to the importance of symmetric functions to the representation theory of the symmetric group, we make the following definition.

\begin{defn}
Given a diagram $D$, the \emph{Schur function} $s_D \in \Lambda$ is given by
\[s_D = \sum c_\lambda^D s_\lambda.\]
\end{defn}

This can also be defined as the image of the character of $S^D$ under the Frobenius characteristic map. When $D$ is a skew Young diagram this definition coincides with the usual definition of a skew Schur function in terms of semistandard tableaux because the $c_\lambda^D$ are again Littlewood-Richardson coefficients. 

In the case when $D$ corresponds to a forest, we can now show the following proposition.

\begin{prop} \label{s-unique}
There is a unique function $s \colon \mathcal F \to \Lambda_\ZZ$ that satisfies the following properties:
\begin{enumerate}
\item For the star $T_n$ with $n$ edges and white center vertex, $s(T_n)=h_n$, the $n$th complete homogeneous symmetric function.
\item For $G_1, G_2 \in \mathcal F$,  $s(G_1+G_2)=s(G_1)s(G_2)$.
\item The function $s$ satisfies the leaf recurrence.
\end{enumerate}
In this case, $s(G)=s_G$, the Schur function associated to $G$.
\end{prop}
\begin{proof}
The proof of uniqueness is identical to that of Proposition~\ref{V-unique}, so it suffices to check that the map $G \mapsto s_G$ satisfies these properties. The first follows because the Schur function corresponding to the partition $(n)$ is $h_n$. The second follows because the analogous statement holds for partitions. The third follows immediately from Corollary~\ref{exact-split}.
\end{proof}

Indeed, we can even say something a little bit stronger: the map $s$ is universal in the following sense.
\begin{thm}~\label{general-leaf}
Let $R$ be a commutative ring with unit, and suppose that $\bar f \colon \{T_n \mid n>0\} \to R$ is any function. Then $\bar f$ can be uniquely extended to a function $f \colon \mathcal F \to R$ that satisfies:
\begin{enumerate}
\item For $G_1, G_2 \in \mathcal F$,  $f(G_1+G_2)=f(G_1)f(G_2)$.
\item The function $f$ satisfies the leaf recurrence.
\end{enumerate}
Moreover, $f$ factors uniquely as $\varphi \circ s$, where $s\colon G \mapsto s_G$ and $\varphi\colon \Lambda_\ZZ \to R$ is a ring homomorphism.
\end{thm}
\begin{proof}
Recall that the $h_n$ are algebraically independent, so that $\Lambda_\ZZ=\ZZ[h_1, h_2, \dots]$. Then setting $\varphi(h_n)=\bar f(T_n)$ defines $\varphi$ uniquely, and the claim follows easily from Proposition~\ref{s-unique}.
\end{proof}

For example, recall that the map $G \mapsto \vol M_G$ satisfies the conditions of Theorem~\ref{general-leaf}. It follows that this map corresponds to a ring homomorphism $\Lambda_\ZZ \to \QQ$. Indeed, this is the map known as \emph{exponential specialization}, which maps a homogeneous function $f \in \Lambda_\ZZ$ of degree $n$ to the coefficient $\frac 1{n!}[x_1x_2\dots x_n]f$. (See, for instance, \cite{EC2}.)

As another example, we may take the map $\Lambda_\ZZ \to \QQ[N]$ that sends a symmetric function $f(x_1, x_2, \dots)$ to its evaluation when $N$ of the variables equal 1 and the rest are 0, sometimes called its \emph{principal specialization}. We find that this has the following combinatorial interpretation.

\begin{prop} \label{m-g}
Let $G \in \mathcal F$. For any nonnegative integer $N$, let $m_G(N)$ be the number of nonnegative integer edge labelings $w\colon E \to \ZZ$ such that for all $v \in V$,
\[\sum_{e \ni v} w(e) \leq \begin{cases}N-1,& \mbox{if $v$ is white, and}\\N-\deg(v),& \mbox{if $v$ is black.}\end{cases}\]

Then $m_G(N)$ is a polynomial in $N$, the map $G \mapsto m_G$ satisfies the conditions of Theorem~\ref{general-leaf}, and $m_G(N)=s_G(\underbrace{1,1,\dots,1}_{N})$.
\end{prop}
\begin{proof}
First suppose $G=T_n$. Then $m_G(N)$ is the number of $n$-tuples of nonnegative integers summing to at most $N-1$, and a simple counting argument shows that this is $\binom{n+N-1}{n}=h_n(\underbrace{1, \dots, 1}_{N}).$

It suffices to check that the map $G \mapsto m_G$ satisfies the conditions of Theorem~\ref{general-leaf}. The first condition is obvious. For the second, we mimic the proof of Proposition~\ref{leaf}. Let $G$, $G_1$, $G_2$, $v_1$, $v_1'$, $v_2$, and $v_2'$ be as before, and let us assume without loss of generality that $v_1$ is white and $v_2$ is black. Let $w$ be a suitable edge weighting for $G$ with $w_i=w(\overline{v_iv_i'})$. If $w_1 \geq w_2$, then let $z$ be the weighting of $G_1$ such that $z(\overline{v_1v_2})=w_2$, $z(\overline{v_1v_1'})=w_1-w_2$, and $z(e)=w(e)$ for all other edges; if $w_1 < w_2$, let $z'$ be the weighting of $G_2$ such that $z'(\overline{v_1v_2})=w_1$, $z'(\overline{v_2v_2'})=w_2-w_1-1$, and $z'(e)=w(e)$ for all other edges. It is simple to check that this gives a bijection between suitable edge weightings for $G$ and suitable edge weightings for either $G_1$ or $G_2$. The result follows.
\end{proof}

The principal specialization gives a polynomial that counts lattice points in a way that is similar to the Ehrhart polynomial of the matching polytope. We shall see a representation-theoretic interpretation of $m_G(N)$ in the next section.


\section{Schur modules}

Just as $V(G)$ has an interpretation in terms of the dimension of a representation of the symmetric group, $m_G(N)$ has an interpretation as a dimension of a representation of the general linear group.

Given a diagram $D$ with $n$ boxes and at most $N$ rows, we define a $GL(N)(=GL(N, \CC))$-module as follows. Let $V=\CC^N$ be the defining representation of $GL(N)$. Then, defining $C_D$, $R_D$, $C(D)$, and $R(D)$ as before, we define the \emph{Schur module} $\mathscr S^D$ by
\[\mathscr S^D = V^{\otimes n} C(D)R(D),\]
where we associate the copies of $V$ in the tensor power to the boxes of the diagram $D$, and $C(D)$ and $R(D)$ act in the obvious way. Note that even if $D$ has more than $N$ rows, we can still make this definition, except that we may find $S^D=0$.

Let $e_1, \dots, e_N$ be the standard basis of $V$. Then we may identify a tableau of shape $D$ with entries at most $N$ with the corresponding tensor product of standard basis elements of $V$, an element of $V^{\otimes n}$. Then $\mathscr S^D$ is spanned by $TC(D)R(D)$, where $T$ ranges over all tableaux of shape $D$ with entries at most $N$.

When $D$ is the Young diagram of a partition $\lambda$ with at most $N$ parts, $\mathscr S^\lambda$ is the irreducible polynomial representation of $GL(N)$ with highest weight $\lambda$. In general, $\mathscr S^D$ is a polynomial representation of $GL(N)$ of degree $n$, and it follows from Schur-Weyl duality \cite{FultonHarris} that
\[\mathscr S^D \cong \bigoplus (\mathscr S^\lambda)^{\oplus c_\lambda^D},\]
where the coefficients $c_\lambda^D$ are the same as those encountered in the decomposition of the Specht module $S^D$.

Recall that the \emph{character} of $\mathscr S^D$ is the trace of the action of the diagonal matrix $\operatorname{diag}(x_1, \dots, x_N)$ on $\mathscr S^D$ as a polynomial in the $x_i$. Then it follows from the fact that $\ch (\mathscr S^\lambda)=s_\lambda$ and the definition of $s_D$ that
\[\ch(\mathscr S^D)=s_D(x_1, \dots, x_N).\]
We therefore find the following.

\begin{prop}
Let $G \in \mathcal F$. Then $m_G(N)=\dim\mathscr S^G$.
\end{prop}
\begin{proof}
Follows from Proposition~\ref{m-g}, since $\dim \mathscr S^G$ is the trace of the identity matrix in $GL(N)$.
\end{proof}

Recall that for Specht modules, Corollary~\ref{corners} gives a description of the restriction of $S^G$ from $\Sigma_n$ to $\Sigma_{n-1}$ by identifying a set of edges that act analogously to corner boxes of a Young diagram. We shall now obtain a similar result for the restriction of $\mathscr S^G$ from $GL(N)$ to $GL(N-1)$ (where $GL(N-1) \subset GL(N)$ acts in the first $N-1$ dimensions and as the identity on the last). This amounts to finding certain subsets of edges that act analogously to horizontal strips of a Young diagram. 

Let $D$ be the diagram of a forest, and fix once and for all a transversal corresponding to an almost perfect matching $U$. Let us assume, as per Proposition~\ref{special-transversal}, that $U$ consists of the boxes $(i,i)$ for $1 \leq i \leq u$ and that $(i,j) \notin D$ for $1 \leq i < j \leq u$ (so that the set $U$ is also equipped with an ordering). We will say that such a diagram is in \emph{standard form}.

\begin{lemma}
Let $D$ be the diagram of a forest with transversal $U$ as above. Then the diagram $D'$ obtained by removing column $u$ from $D$ has a transversal $U'$ corresponding to an almost perfect matching that contains a box from each of the first $u-1$ columns of $D'$ (and possibly others).
\end{lemma}
\begin{proof}
Let $x_0$ be the box $(u,u)$, and choose distinct boxes $x_i$ and $y_i$ recursively such that $y_i$ lies in the same row as $x_i$ (if such a box exists), and $x_{i+1}$ lies in the same column as $y_i$ and in $U$ (if such a box exists). Since $D$ is the diagram of a forest, this must terminate. Then simply replace the set of all $x_i$ in $U$ with the set of all $y_i$ to get $U'$.
\end{proof}

In the definition below, we will write $D'$ for the diagram obtained from $D$ by removing column $u$, and we associate to it the transversal $U'$ as in the lemma. Also, let $D''$ be the diagram obtained from $D$ by removing row and column $u$ and associate to it the transversal $U''=U \backslash \{(u,u)\}$.

\begin{defn}
A subset $Y \subset D$ is called a \emph{horizontal strip} (with respect to $U$) if either $Y=\{(u,u)\} \cup Y'$, where $Y'$ is a horizontal strip of $D'$ with respect to $U'$, or $Y=Y''$, where $Y''$ is a horizontal strip of $D''$ with respect to $U''$. (The empty set is a horizontal strip for any diagram $D$.)
\end{defn}

Here we have abused notation slightly: $D'$ may not be in standard form with respect to $U'$, but it is equivalent to some diagram $E'$ that is in standard form with respect to the image of $U'$. By a horizontal strip of $D'$ we mean the boxes of $D'$ that correspond to a horizontal strip of $E'$.

It is easy to check that the horizontal strips of size 1 exactly correspond to elements of $U$, just as for partitions, a horizontal strip of length 1 is a corner box. It is also easy to check that a horizontal strip can have at most one box in any column but can have more than one box in a row.

Note that the set of horizontal strips of $D$ can vary greatly according to the choices and orderings of almost perfect matchings. We will therefore fix once and for all a single such choice for each diagram $D$, yielding a single set $\mathcal Y(D)$ of horizontal strips.

Although the definition of a horizontal strip is a bit unusual, its value lies in the following result, the analogue of Proposition~\ref{almost-perfect}.

\begin{prop} \label{schur-ineq}
Let $D$ be the diagram of a forest and $N>1$ a positive integer. Then
\[m_D(N)=\sum_{Y \in \mathcal Y(D)} m_{D \backslash Y}(N-1).\]
\end{prop}
\begin{proof}
Let $W_D(N)$ be the set of weightings of $D$ as enumerated by $m_D(N)$ in Proposition~\ref{m-g}. In other words, $W_D(N)$ is the number of labelings of the boxes of $D$ with nonnegative integers such that the sum of each row is at most $N-1$ and the sum of each column is at most $N$ minus the number of boxes in that column.

Note that the conditions on any row or column after the first $u$ are superfluous: since the diagonal corresponds to an almost perfect matching, any row after the first $u$ contains at most one box, and any such box lies in one of the first $u$ columns. Then the condition from the row states that this box is labeled at most $N-1$, which is automatically implied by the condition on its column. (The same argument holds for columns after the first $u$.)

We will construct a bijection $f\colon \bigcup_{Y \in \mathcal Y(D)} W_{D \backslash Y}(N-1)\to W_D(N)$ inductively on the number of boxes of $D$. This bijection will satisfy the following properties: for $w \in W_{D\backslash Y}(N-1)$,
\begin{itemize}
\item $f(w)$ is obtained from $w$ by inserting boxes labeled 0 at the positions of $Y$ and increasing the labels of certain other boxes of $D \backslash Y$ by 1;
\item each of the first $u$ columns will have either a box inserted or a box whose label is increased (not both);
\item no row will contain two boxes with labels that are increased.
\end{itemize}
In particular, this bijection will increase the sum of any row or column by either 0 or 1, and, more specifically, the sum of any of the first $u$ columns will either increase by 1, or else the sum will stay the same but the column will gain a box. Note that this will imply that the image of $f$ is contained in $W_D(N)$.

For the forward direction, let $w \in W_{D\backslash Y}(N-1)$. If $(u,u) \in Y$, then let $z$ be the labeling of $D$ obtained from $Y$ such that $z(u,u)=0$, $z|_{D'}=f(w|_{D'\backslash Y})$ by induction, and $z$ and $w$ agree elsewhere. If $(u,u) \not\in Y$, then let $z$ be such that $z(u,u)=w(u,u)+1$, $z|_{D''}=f(w|_{D'' \backslash Y})$, and $z$ and $w$ agree elsewhere. It is easy to check by induction that all the conditions above are satisfied.

It remains to verify that $f$ is a bijection. Choose $z \in W_D(N)$, and suppose $z(u,u)=0$. Then if $f(w)=z$, then $w$ must be unique. Indeed, first note that we must have $(u,u) \in Y$. Since $z|_{D'} \in W_{D'}(N)$, by induction we can find $f^{-1}(z|_{D'}) \in W_{D' \backslash Y'}(N-1)$ for some horizontal strip $Y'$ of $D'$. Then we must have $Y=\{(u,u)\} \cup Y'$, $w|_{D'\backslash Y'}=f^{-1}(z|_{D'})$, and $w$ and $z$ agree elsewhere on $D \backslash Y$, so $w$ is uniquely determined. We then need to verify that such a $w$ exists, that is, that it lies in $W_{D \backslash Y}(N-1)$. Since $w|_{D'\backslash Y'} \in w_{D' \backslash Y'}(N-1)$, the sum of any of the first $u$ rows is at most $N-2$. Similarly, the sum of any of the first $u-1$ columns is at most $N-1$ minus the length of the column. This also holds for column $u$, for the sum of this column in $z$ was at most $N$ minus the length, and while the sum in $w$ is the same, we removed a box. Thus $w=f^{-1}(z)$ is well-defined and unique in this case.

If $z(u,u) \neq 0$, then we similarly find that we must have $(u,u) \notin Y$ and $f^{-1}(z|_{D''}) \in W_{D''\backslash Y''}(N-1)$ for $Y''=Y$, so $w$ is determined by $w|_{D''\backslash Y''}=f^{-1}(z|_{D''})$, $w(u,u)=z(u,u)-1$, and $w$ and $z$ agree elsewhere. We again need to check that $w \in W_{D \backslash Y}(N-1)$. Since $w|_{D''\backslash Y''} \in W_{D''\backslash Y''}(N-1)$, the sum of any of the first $u$ rows is at most $N-2$ (even row $u$, for it was at most $N-1$ and we subtracted 1 from $(u,u)$). Similarly, the sum of any of the first $u$ columns is at most $N-1$ minus the length of the column, for in each of these columns we either decreased a label by 1 or deleted a box (by the properties above). Thus $w=f^{-1}(z)$ is well-defined and unique in this case as well. The result now follows from Proposition~\ref{m-g}.
\end{proof}

We may now use this result to prove the following analogue of Corollary~\ref{corners}.

\begin{thm} \label{horizontal-strips}
Let $D$ be the diagram of a forest and $\mathscr S^D$ the corresponding $GL(N)$-module. Then
\[\mathscr S^D|_{GL(N-1)} \cong \bigoplus_{Y \in \mathcal Y(D)} \mathscr S^{D \backslash Y}\]
as $GL(N-1)$-modules.
\end{thm}
\begin{proof}
By Proposition~\ref{schur-ineq}, we have that the dimension of both sides of this equation are equal. It therefore suffices to show that the left side contains the right side as a submodule.

For a horizontal strip $Y$, let $T_Y$ be the set of all tableaux $T$ with boxes in $Y$ labeled $N$ and all other boxes labeled at most $N-1$. Define the map $\varphi_Y$ by $\varphi_Y(T)=T$ if $T \in T_Y$ and 0 otherwise. We may identify $\varphi_Y$ with a linear map on $V^{\otimes n}$. Since $GL(N-1)$ stabilizes $e_1, \dots, e_{N-1}$ and fixes $e_N$, $\varphi_Y$ is a $GL(N-1)$-homomorphism.

We construct a linear order $\prec$ on $\mathcal Y(D)$ inductively on the number of boxes of $D$ as follows:
\begin{enumerate}
\item If $(u,u) \in Y$ and $(u,u) \notin Z$, then $Z \prec Y$;
\item If $(u,u) \notin Y, Z$, then $Y''=Y, Z''=Z \in \mathcal Y(D'')$, so write $Z \prec Y$ if and only if $Z'' \prec Y''$;
\item If $(u,u) \in Y,Z$, then $Y'=Y \backslash \{(u,u)\}, Z'=Z \backslash \{(u,u)\} \in \mathcal Y(D')$, so write $Z \prec Y$ if and only if $Z' \prec Y'$.
\end{enumerate}
Let $T \in T_Y$. We claim that $\prec$ has the property that if $Z \prec Y$, then $TC(D)R(D)$ does not contain a term in $T_Z$. In other words, if $Z \prec Y$, then $Tqp \notin T_Z$ for $q \in C_D$, $p \in R_D$. To see this, we check that it holds for the three defining conditions above.

For (1), if $(u,u) \in Y$, then $Tqp$ contains a box labeled $N$ in either row $u$ or column $u$. Hence by the definition of horizontal strips, it cannot lie in $T_Z$ if $(u,u) \notin Z$.

For (2), note that if $(u,u) \notin Y,Z$, then $Y$ and $Z$ do not contain any boxes in row $u$ or column $u$. Thus we may assume that $q$ and $p$ do not affect row $u$ and column $u$. If $Tqp \in T_Z$, then some term of $T|_{D''}\bar{q}\bar{p}$ lies in $T_{Z''}$, where $\bar q$ and $\bar p$ correspond to $q$ and $p$ in $C_{D''}$ and $R_{D''}$. Then we cannot have $Z'' \prec Y''$ and hence neither can we have $Z \prec Y$.

For (3), suppose $Tqp \in T_Z$. Note that $Y$ and $Z$ contain only one box in column $u$, namely $(u,u)$. Then $q$ must send box $(u,u)$ to itself, and $p$ sends $(u,u)$ to another box $(u,v)$. Let $\sigma$ be the transposition in $R_D$ switching $(u,u)$ and $(u,v)$. Then by replacing $p$ by $p\sigma$, we may assume that neither $q$ nor $p$ affects column $u$. Then restricting to $D'$ as in the case above shows this case as well.

It follows that if $T \in T_Y$ and $Z \prec Y$, then $\varphi_Z(TC(D)R(D)) = 0$. Note that the analysis above also shows that if $Tqp \in T_Y$, then $q$ acts as the identity on $Y$, so that we may associate $q$ to an element of $C_{D \backslash Y}$. From this, it follows that $\varphi_Y(TC(D)R(D))$ is a scalar multiple of $TC(D \backslash Y)R(D \backslash Y)$ (by the order of the subgroup of $R_D$ that stabilizes $Y$). 

With this, we are ready to prove the result. Index the horizontal strips such that $Y_1 \prec Y_2 \prec \dots \prec Y_m$, and let $V_i$ be the subrepresentation of $\mathscr S^D|_{GL(N-1)}$ generated by $(T_{Y_i} \cup T_{Y_{i+1}} \cup \dots \cup T_{Y_m})C(D)R(D)$. By the discussion above, $V_{i+1} \subset \ker \varphi_{Y_i}$, while $\varphi_{Y_i}(V_i)$ is generated by $\varphi_{Y_i}(TC(D)R(D))$ for $T \in T_{Y_i}$ and hence is naturally isomorphic to $\mathscr S^{D \backslash Y_i}$. Therefore the successive quotients of $0 \subset V_m \subset V_{m-1} \subset \dots \subset V_1$ contain $\mathscr S^{D \backslash {Y_m}}$, $\mathscr S^{D \backslash {Y_{m-1}}}$, \dots, $\mathscr S^{D \backslash {Y_1}}$ in succession. It follows from the complete reducibility of $GL(N-1)$-representations that we have the desired inclusion, and the result follows.
\end{proof} 

In fact, this shows that if for each $Y \in \mathcal Y(D)$, there exists a set of tableaux $SS(D \backslash Y, N-1)$ such that $SS(D \backslash Y, N-1)C(D \backslash Y)R(D \backslash Y)$ forms a basis of the $GL(N-1)$-module $\mathscr S^{D \backslash Y}$, then we can construct a basis of the $GL(N)$-module $\mathscr S^D$ simply by adding boxes labeled $N$ to the tableaux in $SS(D \backslash Y, N-1)$ along $Y$ for all $Y$. This leads us to the following definition.

\begin{defn}
Let $D$ be the diagram of a forest, and let $T$ be a tableau of shape $D$ with entries at most $N$. We say that $T$ is \emph{semistandard} if the boxes labeled by $N$ in $T$ form a horizontal strip in $\mathcal Y(D)$, and the tableau formed from $T$ by removing all boxes labeled $N$ is also semistandard. (The empty tableau is semistandard by default.) We say that $T$ is \emph{standard} if it is semistandard and contains only the entries $1, \dots, n$, each exactly once. 
\end{defn}

(Note that since horizontal strips of length 1 on $D$ precisely correspond to edges in an almost perfect matching of the graph of $G$, this definition agrees with the definition of standard labelings given at the end of Section 2.)

It is easy to see that a semistandard tableau is equivalent to a sequence of diagrams $\emptyset = D^{(0)} \subset D^{(1)} \subset \dots \subset D^{(N)}=D$ such that $D^{(i)} \backslash D^{(i-1)} \in \mathcal Y(D^{(i)})$ for $1 \leq i \leq N$. We denote the set of semistandard tableaux of shape $D$ with entries at most $N$ by $SS(D,N)$ and the set of all semistandard tableaux of shape $D$ by $SS(D)$.

\begin{prop}
Let $D$ be the diagram of a forest. Then
\[s_D(x_1, x_2, \dots)=\sum_{T \in SS(D)} x_1^{\alpha_1}x_2^{\alpha_2}\cdots,\]
where  $\alpha_i$ is the number of occurrences of $i$ in $T$. In particular, $\dim(S^D)=V(D)$ is the number of standard tableaux of shape $D$.
\end{prop}
\begin{proof}
From the discussion after the proof of Theorem~\ref{horizontal-strips}, we find that $TC(D)R(D)$, for $T \in SS(D,N)$, forms a basis for $\mathscr S^D$. Thus, since $s_D(x_1, \dots, x_N)$ is the character of $\mathscr S^D$, and $\operatorname{diag}(x_1, \dots, x_N)$ acts on $TC(D)R(D)$ by multiplication by $x_1^{\alpha_1}x_2^{\alpha_2}\cdots$, the result follows easily.
\end{proof}

This therefore provides a combinatorial description of the coefficients of $s_D$. 

\section{Conclusion}

In this paper, we have shown that there is a surprising relationship between matchings of a forest and the representation theory of the symmetric and general linear groups. In particular, we have shown how properties of the Specht and Schur modules of a forest can arise from considering its matchings and its matching polytope.

However, there is still one piece of this picture that remains. Although we have given a method of determining combinatorially the Schur function $s_G$ for $G \in \mathcal F$, we have not yet explicitly given a combinatorial rule for the coefficients $c_\lambda^G$ akin to the Littlewood-Richardson rule. In other words, we would like to identify $c_\lambda^G$ as enumerating a certain type of combinatorial object. Since diagrams corresponding to forests are not in general ``percentage-avoiding,'' this would give a new class of diagrams for which we can determine $c_\lambda^D$ that is not already covered by existing results.

Even though the proofs above are fairly short, in some sense they remain a bit mysterious. In particular, it seems almost coincidental that both $\dim S^G$ and $V(G)$ happen to satisfy the leaf recurrence, which serves as the basis for all the results above. First, we can ask whether there are any other interesting functions that satisfy the leaf recurrence, and if so, what their role in this story is. We can also ask whether there is a more natural explanation for the phenomena presented here, and in particular whether it is possible to formulate similar results that are true for more general classes of diagrams.

If a more general picture does in fact exist, it will necessarily be more complicated than in the case of forests. For instance, it is not true in general that the normalized volume of the matching polytope of a bipartite graph is equal to the dimension of its Specht module. (As an example, if $G$ is a cycle of length 4, then $\dim S^G=2$ but $V(G)=4$.) One might then ask whether there exists a nice combinatorial interpretation for the volume of the matching polytope of a general bipartite graph, or indeed of a general graph.

It is also not true that for a general diagram $D$ one can define ``corner boxes'' and ``horizontal strips'' that satisfy the properties of Corollary~\ref{corners} and Theorem~\ref{horizontal-strips}. (The smallest counterexample corresponds to the cycle of length 6.) However, this is possible for many diagrams, including all northwest shapes \cite{ReinerShimozono}. It would therefore be interesting to classify those diagrams for which one can make such definitions.

\section{Acknowledgments}
I would like to thank Alexander Postnikov and Richard Stanley for their input in the development of this paper. This material is based upon work supported under a National Science Foundation Graduate Research Fellowship. 

\bibliography{matching3}
\bibliographystyle{plain}

\end{document}